\documentclass[11pt]{amsart} 
\usepackage{verbatim, latexsym, amssymb, amsmath,color}
\usepackage{epsfig}
\usepackage{hyperref}
\usepackage{float}
\usepackage{graphicx}
\usepackage{sidecap}

\def\Xint#1{\mathchoice
   {\XXint\displaystyle\textstyle{#1}}%
   {\XXint\textstyle\scriptstyle{#1}}%
   {\XXint\scriptstyle\scriptscriptstyle{#1}}%
   {\XXint\scriptscriptstyle\scriptscriptstyle{#1}}%
   \!\int}
\def\XXint#1#2#3{{\setbox0=\hbox{$#1{#2#3}{\int}$}
     \vcenter{\hbox{$#2#3$}}\kern-.5\wd0}}

\def\dashint{\Xint-}

\newtheorem{thm}{Theorem}[subsection]

\newtheorem{lemm}[thm]{Lemma}
\newtheorem{cor}[thm]{Corollary}

\newtheorem{prop}[thm]{Proposition}
\newtheorem{thmA}{Theorem}[subsection]

\newtheorem{thmB}{Theorem}[subsection]

\theoremstyle{remark}

\newtheorem*{rmkA}{Remark}

\theoremstyle{definition}

\numberwithin{equation}{subsection}

\title{On the min-max width of unit volume three-spheres}
\author{Lucas Ambrozio and Rafael Montezuma}
\address{L. Ambrozio: Institute for Advanced Study \\ Princeton NJ
08540 USA}
\email{lambrozio@ias.edu}

\address{R. Montezuma: Mathematics Department, Princeton University \\ Fine Hall, Washington Road \\ Princeton NJ 08544-1000 USA}
\email{rcabral@princeton.edu}

\begin{document}
\maketitle

\begin{abstract}
	How large can be the width of Riemannian three-spheres of the same volume in the same conformal class? If a maximum value is attained, how does a maximising metric look like? What happens as the conformal class changes? In this paper, we investigate these and other related questions, focusing on the context of Simon-Smith min-max theory.
\end{abstract}

\section{Introduction}

\indent The three-sphere $S^3$, regarded as the set of points in the four-dimensional Euclidean space that are at distance one from the origin, is swept out by a very simple family of two-spheres, that starts at the south pole and finishes at the north-pole. This family is the standard sweep-out of $S^3$ given by the horizontal two-spheres
\begin{equation*}
	\overline{\Sigma}_t =\{(x^1,x^2,x^3,x^4)\in S^3;\, x^4=t\} \quad \text{for each} \quad t\in [-1,1].
\end{equation*}
\noindent New families of two-spheres sweeping out $S^3$ can be produced from the standard one by taking a smooth one-parameter family of diffeomorphisms $F_t : S^3 \rightarrow S^3$,  $t\in [-1,1]$, that are isotopic to the identity map, and defining $\Sigma_t=F_{t}(\overline{\Sigma}_t)$ for each $t$ in $[-1,1]$. Denoting by $\overline{\Lambda}$ the set of all these sweep-outs $\{\Sigma_t\}$, the \textit{Simon-Smith width} of a Riemannian three-sphere $(S^3,g)$ is defined as the geometric invariant
\begin{equation*}
	W(S^3,g) = \inf_{\{\Sigma_t\}\in \overline{\Lambda}} \left( \max_{t\in[-1,1]}area(\Sigma_t,g) \right).
\end{equation*}
\indent A motivation for the introduction of the min-max invariant $W(S^3,g)$ is the following: sweep-outs can be regarded as non-trivial loops in the space of two-spheres in $S^3$, and if the insights from finite dimensional Morse theory were true for the area functional on the space of two-spheres in $(S^3,g)$, then the expectation would be that this number corresponds to the area of a critical point, \textit{i.e.} a minimal two-sphere in $(S^3,g)$, with Morse index at most one. This approach was successful in showing the existence of embedded minimal two-spheres in any Riemannian three-sphere. More precisely, Simon and Smith \cite{SimSmi} proved that, for every Riemannian metric $g$ on $S^3$, the number $W(S^3,g)$ is realised as the total area of a finite collection of disjoint embedded minimal two-spheres (when their areas are counted with appropriate integer multiplicities), and Marques and Neves \cite{MarNev-index} showed that these two-spheres can be constructed so that the sum of their Morse indices is at most one. An excellent introduction to this variational method of producing minimal surfaces in three-manifolds is given in the survey by Colding and De Lellis \cite{ColdeL}. \\
\indent In this paper, we study the \textit{normalised (Simon-Smith) width}, 
\begin{equation*}
	\frac{W(S^3,g)}{vol(S^3,g)^{\frac{2}{3}}},
\end{equation*}
among metrics $g$ that belong to certain subsets of the space of Riemannian metrics on $S^3$. Of particular interest are the sets $[g_0]$ of metrics that are conformal to a given metric $g_0$. For example, we will seek upper bounds for the normalised width on these sets, and describe necessary conditions satisfied by metrics that attain a local maximum.

\subsection{Upper bounds for the normalised width} There are different approaches to a min-max variational theory of the area functional along the lines alluded above, each possibly defining a different min-max invariant, and possibly producing different minimal surfaces. Arguably, the most general and flexible theory is the pioneering one, developed by Almgren \cite{Alm} and Pitts \cite{Pit} in the context of Geometric Measure Theory. The Simon-Smith width $W(S^3,g)$ is bounded from below by the Almgren-Pitts width associated to the fundamental class of $S^3$ \cite{Pit}, because the latter min-max procedure allows more general sweep-outs, not necessarily by smooth embedded two-spheres.\\
\indent Considering examples of Riemannian metrics constructed by Burago and Ivanov \cite{BurIva}, Guth showed that there are unit volume metrics on $S^{3}$ with arbitrarily large Almgren-Pitts width (see \cite{Gut}, Appendix A, and also \cite{Sab-ricci}, Section 7, for further results in this direction). Thus, no upper bound for the normalised width can be valid for all Riemannian metrics on $S^3$. \\
\indent In \cite{Tre}, Treibergs proved the existence of a constant $C>0$ such that, for any metric $g$ on $S^3$ that is the restriction of the Euclidean metric to a convex hypersurface in the four-dimensional Euclidean space, the following inequality holds:
\begin{equation*}
	W(S^3,g) \leq C vol(S^3,g)^{\frac{2}{3}}.
\end{equation*} 
\indent More recently, Glynn-Adey and Liokumovich \cite{GlALio} proved that, for every conformal class of Riemannian metrics on $S^3$ (actually, on any closed manifold), there exists a constant that bounds from above the Almgren-Pitts width of every unit volume metric within this class. The precise statement of their result is actually rather explicit about the geometric dependence of the constant they obtain, see Theorems 5.1 therein; in particular, their estimate implies that all unit volume positive Ricci curvature metrics have their Almgren-Pitts widths uniformly bounded from above as well. The latter bound was also proven by Sabourau \cite{Sab-ricci}. We remark that the Almgren-Pitts width of a three-sphere with positive Ricci curvature is precisely the least possible area of an embedded minimal surface (see \cite{Son}, Theorem 15), while the Simon-Smith width of such three-sphere is equal to the least possible area of an embedded minimal two-sphere (see \cite{MarNev-Duke}, proof of Theorem 3.4). \\
\indent In contrast with the above results, we will check that the normalised Simon-Smith width is unbounded among positive Ricci metrics on $S^3$ (see Theorem \ref{thm-bergermetrics}). On the other hand, at the moment we do not know if the normalised (Simon-Smith) width of metrics in a given conformal class on $S^3$ is bounded from above. Nevertheless, we gathered some evidence that at least certain conformal classes may indeed have Riemannian metrics that are local maxima of the normalised width. And we found that the round metric is indeed a locally maximising metric in its conformal class, a result that we now describe in details.

\subsection{The width of the round three-sphere} The canonical round metric $g_1$ on $S^3$, with constant sectional curvature one, has volume equal to $2\pi^2$ and width equal to $4\pi$ (= area of a totally geodesic equator). Thus,
\begin{equation*}
	\frac{W(S^3,g_1)}{vol(S^3,g_1)^{\frac{2}{3}}} = \sqrt[\leftroot{-1}\uproot{2}\scriptstyle 3]{\frac{16}{\pi}}.
\end{equation*}
\noindent Our first result shows that $g_1$ is a local maximum of the normalised width inside its conformal class. More precisely:
\begin{thm} \label{thm-roundsphere}
	If $g$ is a Riemannian metric on $S^3$ that is conformally flat and has positive Ricci curvature, then
	\begin{equation*}
		W(S^3,g) \leq \sqrt[\leftroot{-1}\uproot{2}\scriptstyle 3]{\frac{16}{\pi}} Vol(S^3,g)^{\frac{2}{3}}.
	\end{equation*}
	Moreover, equality holds if and only if $g$ has constant sectional curvature.
\end{thm}
\indent Our proof relies on estimates for minimal two-spheres with Morse index one in Riemannian three-manifolds. Another ingredient of the argument is the Yamabe flow, introduced by Hamilton and studied in the above setting by Chow \cite{Cho}.\\
\indent We have not investigated further whether an analogous result holds for metrics on the $n$-sphere, $n> 3$, and an appropriate notion of width (\textit{e.g.} the Almgren-Pitts width associated to the fundamental class of $S^n$). On the other hand, and perhaps surprisingly, the analogous statement is false in the case of $S^2$, where all metrics are conformal to the round metric as a consequence of the Uniformisation Theorem. In fact, there are metrics with positive Gaussian curvature on $S^2$ that have larger width than a round two-sphere with the same volume (for an explicit construction, see \cite{ManSch}, Section 4). These metrics are close to the Calabi-Croke metric, a flat metric on $S^2$ with three conical singularities, nicely described with pictures in \cite{Bal}, Section 2.2 (conjecturally, the Calabi-Croke metric attains the supremum of the length of a shortest closed geodesic among all unit volume Riemannian metrics on $S^2$, as discussed in \cite{Cro}, \cite{Bal}, \cite{Sab-calabi}). Although the round metric is not an absolute maximum, it is indeed a local maximum of the normalised width. The proof of this statement was given by Abbondandolo, Bramham, Hryniewicz and Salom\~ao \cite{AbbBraHrySal}, and involves an interesting combination of techniques from both  symplectic and Riemannian geometry. To complete this picture, we mention that the rigidity statement does not hold either in two dimensions: by a theorem of Weinstein \cite{Wei}, all Zoll metrics on $S^2$ have the same normalised width as the round metric, and the round metric admits non-trivial deformations by Zoll metrics (see \cite{Gui}, \cite{AbbBraHrySal} and references therein).  \\
\indent We remark that Souam \cite{Sua} proved that closed embedded index one minimal surfaces in a conformally flat three-sphere with positive Ricci curvature are two-spheres. But the Almgren-Pitts width of positive Ricci metrics is indeed realised as the area of a closed minimal surface of this type (see \cite{Zho} Theorem 1.1). Thus, in the setting of Theorem \ref{thm-roundsphere}, the Almgren-Pitts width coincides with the Simon-Smith width. \\
\indent Finally, we point out that, as all admissible sweep-outs have a leaf that bound a region of the three-sphere whose volume is any given fraction of the total volume of $(S^3,g)$, the width is an upper bound for the maximum value of the isoperimetric profile of $(S^3,g)$. For metrics with positive Ricci curvature, the profile is concave and its maximum is attained by a surface dividing $(S^3,g)$ into two pieces of equal volume (see \cite{Ros}, Theorem 18). Thus, Theorem \ref{thm-roundsphere} has the following direct consequence:
\begin{cor}
	The isoperimetric profile of a conformally flat Riemannian metric with positive Ricci curvature on the three-sphere is bounded from above by the area of the equators of the round three-sphere with the same volume, and equality holds if and only if the metric is round itself.
\end{cor}

\subsection{The width of homogeneous three-spheres} Going beyond conformally flat metrics, we study the Simon-Smith width of homogeneous metrics on $S^3$, \textit{i.e.} metrics whose isometry group acts transitively. When metrics that differ by scaling are identified, the homogeneous metrics on $S^3$ form a two-parameter family. It does not seem to be a trivial task to compute the width of homogeneous three-spheres in terms of these parameters, but we do know that the following general sharp inequality must hold:
\begin{thm}
	Let $g$ be a homogeneous Riemannian metric on the three-sphere with constant scalar curvature 6. Then
		\begin{equation*}
			W(S^3,g) \leq 4\pi,		
		\end{equation*}		 	
		and equality holds if and only if $(S^3,g)$ has constant sectional curvature one.
\end{thm}
\indent We remark that some homogeneous metrics with positive scalar curvature do not have positive Ricci curvature, and therefore the above statement is not an immediate consequence of the main theorem in \cite{MarNev-Duke}, but rather of their methods of proof. In fact, the Ricci flow preserves the isometry group of the initial metric, and homogeneous three-spheres contains no embedded stable minimal two-spheres. In particular, the arguments of \cite{MarNev-Duke} can be used to prove the above statement, as explained in more details in Section \ref{Sec-4}.  \\
\indent In any case, if we restrict our attention to the specific family of Berger metrics (which contains metrics with negative, zero and positive scalar curvature), it is actually possible to derive a formula for their widths.  As a result of a qualitative analysis performed in Section \ref{Sec-4}, we conclude that:

\begin{thm} \label{thm-bergermetrics}
	The round metric is a strict local minimum of the normalised width in the one-parameter family of Berger metrics on the three-sphere. Moreover, the normalised width is not uniformly bounded from above among Berger metrics, not even among those that have positive Ricci curvature. 
\end{thm}

\indent We expect all homogeneous metrics on $S^3$ to be local maxima of the normalised width among metrics in their respective conformal classes, or at least critical points in an appropriate sense. The reason for this conjecture is that we were able to check that homogeneous metrics satisfy necessary conditions that any local maximum must satisfy, which we describe in the next Subsection.

\subsection{Locally maximising metrics} Despite our positive result about conformally flat metrics, it is still unclear to us whether a (local) maximum of the normalised width exists in every conformal class on $S^3$. Most likely, this existence question is delicate and one should allow singular metrics when discussing it. Nevertheless, we will show that a Riemannian metric on the three-sphere that is a local maximum of the normalised width inside a conformal class must have a very peculiar geometry, which gives us a hint about where should we look out for them. 
\begin{prop} \label{prop-necessary-equidistribution}
	If $g$ is a Riemannian metric on the three-sphere that is a local maximum of the normalised Simon-Smith width inside its conformal class, then there exists a sequence $\{\Sigma_{i}\}$ of embedded minimal two-spheres in $(S^3,g)$, with Morse index at most one and area at most $W(S^3,g)$, such that, for every continuous function $f$ on $S^3$,
	\begin{equation} \label{eq-introduction-equidistribution}
		\lim_{k\rightarrow +\infty} \frac{1}{\sum_{i=1}^{k}area(\Sigma_i,g)} \sum_{i=1}^{k} \int_{\Sigma_{i}} fdA_{g} = \frac{1}{vol(S^3,g)}\int_{S^3} f dV_{g}.
\end{equation}	 
	If, moreover, $(S^3,g)$ contains no stable minimal two-spheres of area less than or equal to $W(S^3,g)$, then the sequence $\{\Sigma_i\}$ can be taken so that each $\Sigma_i$ has Morse index one and area equal to $W(S^3,g)$.
\end{prop}

\indent Thus, locally maximising metrics contains infinitely many minimal two-spheres, of bounded index and area, that are evenly distributed on $S^3$ in the sense described by \eqref{eq-introduction-equidistribution}. By Sharp's Compactness Theorem \cite{Sha}, it follows that locally maximising metrics are not bumpy in the sense of White \cite{Whi2}, confirming the intuition that such metrics must be very rare. \\
\indent We propose two different proofs of the above result, see Propositions \ref{prop-star-equidistribution} and \ref{prop-equidistribution}. The first proof only works in the less general situation where $(S^3,g)$ contains no stable minimal two-spheres of area bounded from above by $W(S^3,g)$. This extra condition guarantees the existence of certain optimal sweep-outs that are necessary for the argument, which was inspired by the work of Fraser and Schoen \cite{FraSch} on Steklov eigenvalues (see also \cite{Mac}). The second proof, which holds in general and can be translated to the case of other min-max invariants (\textit{e.g.} the Almgren-Pitts width), was influenced by the work of Marques, Neves and Song \cite{MarNevSon}, who established the equidistribution of closed smooth minimal hypersurfaces in generic closed Riemannian manifolds of dimension $n=3,4,5,6,7$; notice the similarity between \eqref{eq-introduction-equidistribution} and Formula (1) in \cite{MarNevSon}. We hope that presenting both proofs together in this paper will not be a pointless exercise, and will at least offer some further insight into similar maximisation problems. \\
\indent All homogeneous Riemannian metrics on $S^3$ satisfy the conclusions of Proposition \ref{prop-necessary-equidistribution}, see Section \ref{Sec-4} for further discussion. \\
\indent The second necessary condition that we deduce compares the normalised width of a locally maximising metric $g_0$ in $[g_{0}]$ to the \textit{Yamabe invariant} of its conformal class,
\begin{equation*}
	\mathcal{Y}(S^3,[g_0]) = \inf_{g\in [g_0]} \frac{\int_{S^3} R_{g}dV_{g}}{vol(S^3,g)^{\frac{1}{3}}},
\end{equation*}  
where $R_g$ denotes the scalar curvature of the metric $g$.

\begin{prop} \label{prop-intro-Yamabe}
	Let $[g_0]$ be a conformal class of Riemannian metrics on the three-sphere with positive Yamabe invariant. Assume that $g_0$ contains no stable minimal surfaces of area smaller than or equal to $W(S^3,g_0)$. If $g_0$ is a local maximum of the normalised width in $[g_0]$, then 
	\begin{equation*}
		W(S^3,g_0) \leq \frac{24\pi}{\mathcal{Y}(S^3,[g_0])}vol(S^3,g_0)^{\frac{2}{3}}.
	\end{equation*}	 
\end{prop}

\indent See Subsection 5.3 for a slightly more general statement, and also a further discussion on homogeneous metrics. \\
\indent The proposition above suggests a non-trivial relation between two different conformal invariants, namely the conformal width $W(S^3,[g])$ and the Yamabe invariant $\mathcal{Y}(S^3,[g])$. This matter seems to be of some interest, and is perhaps independent of the problem of existence of maximising metrics.

\subsection{Related results, and structure of the paper} In the aforementioned work \cite{MarNev-Duke}, Marques and Neves proved sharp estimates for the width of three-spheres with positive Ricci curvature in terms of a lower bound for the scalar curvature (the hypothesis on the Ricci curvature cannot be weakened to a hypothesis on the scalar curvature, see \cite{Mon}). Their proof involves the analysis of how does the Simon-Smith width evolve under the Ricci flow. In a previous work, Colding and Minicozzi \cite{ColMin} studied the evolution of a similar min-max invariant, using it to prove finite extinction time of Ricci flows on closed orientable prime non-aspherical three-manifolds. \\
\indent Gromov \cite{Gro} advanced the extremely fruitful idea that widths can be thought as non-linear analogous of the spectrum of linear operators, \textit{e.g.} the Laplace operator. This guiding analogy has lead to many ground-breaking results in the field of minimal surfaces in recent years, which we dare not to survey here, for it would lead to a long digression; instead, we just mention one of them, whose statement closely resembles a classical result about the spectrum of the Laplace operator: Marques, Neves and Liokumovich proved that the (Almgren-Pitts) $p$-widths obey a Weyl Law \cite{MarNevYev}. \\
\indent The proof of the upper bounds for the normalised Almgren-Pitts width in conformal classes mentioned in Subsection 1.1 was also inspired by this analogy, according to its authors \cite{GlALio}. Also in our paper, we have adapted to our purposes some techniques used to study extremising metrics of (normalised) Laplace eigenvalues on closed surfaces (\textit{e.g.} \cite{Nad}) and Steklov eigenvalues on compact surfaces with boundary (\textit{cf}. \cite{FraSch}, Section 2). \\
\indent There are also parallels between some of our results on widths and theorems about \textit{systoles} ($=$ the length of the shortest \textit{non-trivial} closed geodesic). For example, Bavard \cite{Bav} deduced a necessary (and sufficient) condition for a maximising metric of the normalised systole in a conformal class, which is formally very similar to the equidistribution criterion for maximising metrics of the normalised width we prove. \\
\indent In a companion paper \cite{AmbMon}, we investigated a related geometric invariant, the \textit{two-systole} of three-di\-men\-sional real projective spaces $(\mathbb{RP}^3,g)$. If $S^3$ is endowed with a Riemannian metric that admits no stable minimal two-spheres and is the pull-back of a Riemannian metric $g$ on $\mathbb{RP}^{3}$ by the standard projection $\pi : S^3 \rightarrow \mathbb{RP}^3$, then the value of the two-systole of $(\mathbb{RP}^3,g)$ is an upper bound for half the value of width of $(S^{3},\pi^{*}g)$. Equality holds for homogeneous metrics, and we prove sharp inequalities for the normalised two-systole in the conformal classes of homogeneous metrics on $\mathbb{RP}^3$. \\
\indent The main motivation of this work was to advance our understanding of normalised widths as functionals on the space of Riemannian metrics. Our choice to focus on three-spheres and a very particular min-max invariant, the Simon-Smith width, is due in part to presentation reasons. We hope that the reader will recognise that many arguments presented here do extend, with not that much extra effort, to other notions of width (particularly Proposition \ref{prop-necessary-equidistribution}). More importantly, we hope that some of the questions raised here will be answered in the near future. \\

\indent For the reader convenience, we now describe the structure of the paper. In Section \ref{Sec-2}, we prove technical lemmas about the existence of optimal sweep-outs and about derivatives of the width along one-parameter families of metrics. In Section \ref{Sec-3}, we review some properties of the Yamabe flow and prove Theorem \ref{thm-roundsphere}. In Section \ref{Sec-4}, we explain our results about homogeneous metrics. In Section \ref{Sec-5}, we deduce some necessary conditions for local maxima of the normalised width in their conformal classes. Finally, in the Appendix, we state a general compactness result for minimal two-spheres that is used repeatedly, and prove a key abstract result needed in Section \ref{Sec-5}. \\

\noindent \textbf{Acknowledgements}: This project kicked off during a visit of the first author to Princeton University in April 2018. L.A. is grateful to Fernando Marques for the invitation to this academic visit. At the time, and while the main results were obtained, L.A. was a Research Fellow at the University of Warwick, supported by the EPSRC Programme Grant `Singularities of Geometric Partial Differential Equations', reference number EP/K00865X/1. L.A would like to thank Peter Topping for his kind interest in this work.

\section{Optimal sweep-outs and derivatives of the width} \label{Sec-2}

\subsection{The condition $(\star)$ and optimal sweep-outs} An important theme within a min-max theory for the area functional is to identify which critical points realise the value of the width. In certain cases, this amounts to identify optimal sweep-outs containing a given critical point as a leaf. \\
\indent In \cite{MarNev-Duke}, the authors identified a geometric condition that guarantees that every minimal two-sphere in $(S^3,g)$ belong to an optimal sweep-out, which moreover has nice further properties (\textit{cf}. Theorem 3.4 therein). This condition is satisfied, for example, if $g$ has positive Ricci curvature. \\
\indent  A slightly less stringent condition is the following: we will say that a Riemannian three-sphere $(S^3,g)$ satisfies property $(\star)$ when it contains no stable embedded minimal two-sphere with area less than or equal to $W(S^3,g)$. \\
\indent A useful property of the condition $(\star)$ is that it is stable under perturbations. More precisely:
\begin{lemm} \label{lemm-openness}
	The set of Riemannian metrics on $S^3$ that satisfy property $(\star)$ is an open subset of the space of Riemannian metrics on $S^3$. 
\end{lemm}
\begin{proof}
	We prove that the complement of this set is closed. Let $\{g_{k}\}$ be a sequence of Riemannian metrics on $S^3$, converging smoothly to a Riemannian metric $g$, such that for each positive integer $k$ there exists  an embedded stable minimal two-sphere $\Sigma_k$ in $(S^3,g_k)$ with $area(\Sigma_k,g_k)\leq W(S^3,g_k)$. Since the widths of $(S^3,g_k)$ converges to the width of $(S^{3},g)$, the areas of the two-spheres $\Sigma_k$ are uniformly bounded. By the Compactness Theorem \ref{thm-compactness}, the sequence $\{\Sigma_k\}$ has a subsequence that converges smoothly to an embedded stable minimal two-sphere $\Sigma$ in $(S^{3},g)$, whose area is at most $\limsup area(\Sigma_k,g_{k}) \leq \lim {W(S^3,g_k)}=W(S^{3},g)$.
\end{proof}
\indent The other useful property is the existence of optimal sweep-outs. We outline below the proof of a version of Theorem 3.4 in \cite{MarNev-Duke}, following the same ideas (\textit{cf}. \cite{MazRos}, Proposition 16).
\begin{lemm} \label{lemm-optimal-sweepouts}
	If $g$ is a Riemannian metric on the three-sphere that satisfies property $(\star)$, then there exists an embedded index one minimal two-sphere $\Sigma_0$ in $(S^3,g)$ such that
	\begin{equation*}
		area(\Sigma_0,g) = W(S^3,g).
	\end{equation*} 
	Moreover, $\Sigma_0$ belongs to a sweep-out $\{\Sigma_s\}_{s\in[-1,1]}$ such that
	\begin{itemize}
		\item[$i)$] $area(\Sigma_s,g)< area(\Sigma_0,g)$ for all $s\neq 0$ in $[-1,1]$.
		\item[$ii)$] the family $\Sigma_{s}$ is smooth in a neighbourhood of $s=0$.
		\item[$iii)$] the function $F(s)=area(\Sigma_s,g)$, $s\in [-1,1]$, satisfies $F''(0) < 0$.
	\end{itemize}
\end{lemm}
\begin{proof} Simon-Smith Min-max Theorem \cite{SimSmi} produces a collection of disjoint embedded minimal two-spheres whose total area, counted with multiplicity, is equal to $W(S^3,g)$. By property $(\star)$, none of them can be stable. Moreover, there can be at most one of them: if not, there would be two disjoint unstable minimal two-spheres bounding a mean convex region $\Omega$ in $(S^3,g)$, and Meeks-Simon-Yau Theorem \cite{MeeSimYau} would produce an embedded stable minimal two-sphere, lying inside $\Omega$, with area at most equal to the area of one of the boundary components of $\Omega$, a number that is at most equal to $W(S^3,g)$.\\
	\indent Let then $\Sigma_0$ be the unstable embedded minimal two-sphere produced above. Notice that its area is at most equal to $W(S^3,g)$ (up to this point, we do not know yet if the Min-max Theorem has produced $\Sigma_0$ with multiplicity $>1$). We now argue that $\Sigma_0$ belongs to a sweep-out satisfying the extra properties $i)$, $ii)$ and $iii)$ above. In a neighbourhood of $\Sigma_0$, the sweep-out is obtained by pushing $\Sigma_0$ in the normal direction with speed given by the first Jacobi eigenfunction, so to guarantee that properties $ii)$ and $iii)$ hold. Using the min-max procedure for strictly mean convex regions devised by Marques and Neves in \cite{MarNev-Duke}, Theorem 2.1, we now just need to observe that, if we could not extend this family to sweep-out satisfying $i)$, we would be able to produce another minimal surface $\Sigma$ disjoint from $\Sigma_0$, and then minimising area in the isotopy class of $\Sigma_0$ inside the region bounded by these two surfaces would produce a contradiction exactly as before. \\
	\indent Since $\Sigma_0$ belongs to an admissible sweep-out satisfying $i)$, we conclude \textit{a posteriori} that $area(\Sigma_0,g)=W(S^3,g)$. It remains only to show that the Morse index of $\Sigma_0$ is one. If this were not the case, the extra direction into which we could strictly decrease the area to second order would allow us to construct a sweep-out violating the definition of the width, exactly as explained at the end of the proof of Theorem 3.4 in \cite{MarNev-Duke}. The Lemma is now proved.
\end{proof}

\subsection{The width of varying metrics} The width is a continuous but non-differentiable function on the space of Riemannian metrics. However, it is possible to differentiate it at most points along paths of metrics. This Subsection is dedicated to some key technical lemmas regarding derivatives of the width.

\begin{lemm}\label{lemm-star-derivative-width}
	If $g(t)$, $t\in [a,b]$, is a smooth family of metrics on $S^3$, then the function 
	\begin{equation*}
		W: t\in [a,b] \mapsto W(S^3, g(t)) \in \mathbb{R}
	\end{equation*} 
	is Lipschitz continuous. \\
	\indent If, moreover, $W$ is differentiable at a point $t_0\in (a,b)$ where the metric $g(t)$ satisfies property $(\star)$, then there exists an embedded index one minimal two-sphere $\Sigma$ in $(S^3,g(t_0))$ such that 
	\begin{equation*}
		area(\Sigma,g(t_0)) = W(S^3,g(t_0)) \,\,\, \text{and} \,\,\,	W^{\prime}(t_0) = \frac{1}{2} \int_{\Sigma}  \text{Tr}_{(\Sigma, g(t_0))}(\partial_t g(t_0))dA_{g(t_0)}. 
	\end{equation*} 
\end{lemm}

\begin{proof}
	The metrics $g(t)$ are all comparable to a fixed metric, say $g(0)$, because the family is varying smoothly. It is then straightforward to show that $W$ is Lipschitz continuous (\textit{cf}. Lemma 4.1 in \cite{MarNev-Duke}, for example). By Rademacher Theorem, $W$ is differentiable at almost every point $t$ in $[a,b]$. \\
	\indent Let us assume that $W$ is differentiable at $t_0\in (a,b)$, and that $g(t_0)$ satisfies property $(\star)$. In order to compute the derivative $W'(t_0)$, we use an optimal sweep-out $\{\Sigma_s\}$ associated to the index one minimal two-sphere $\Sigma_0$ constructed in Lemma \ref{lemm-optimal-sweepouts}. Using the notation of that Lemma, we define the function
	\begin{equation*}
		F: (s,t) \in [-1,1]\times [a,b] \mapsto area(\Sigma_s, g(t)) \in \mathbb{R}.
	\end{equation*}
Because of the properties of $\{\Sigma_s\}$, the function $F$ is differentiable in a neighbourhood of the point $(0,t_0)$, with $F(0, t_0) = \max F(\cdot,t_0) = W(t_0)$, $F_{s}(0,t_0)=0$ and $F_{ss}(0, t_0)<0$ (here, the subscripts denote the partial derivative of $F$ with respect to the variable $s$). Moreover, 
	\begin{equation}\label{eq1-proof1}
	\nabla F(0, t_0) = \bigg( 0, \frac{1}{2}\int_{\Sigma} \text{Tr}_{(\Sigma, g(t_0))}(\partial_t{g}(t_0)) dA_{g(t_0)} \bigg).
	\end{equation}
\indent 	We claim that there exists a smooth function $s = s(t)$, defined on a neighborhood of $t = t_0$, such that $s(t_0) = 0$ and
	\begin{equation}\label{eq2-proof1}
	F(s(t), t) = \max \{F(s, t);\, s \in [-1,1]\}.
	\end{equation}
	In order to verify this, notice first that $F_s(0, t_0)= 0$, because $s = 0$ is the maximum of $F(\cdot, t_0)$. Since $F_{ss}(0, t_0)<0$, the implicit function theorem guarantees the existence of a smooth function $s : (t_0-\eta, t_0+\eta)\rightarrow (-\varepsilon, \varepsilon)$ such that $s(t_0)=0$ and $F_s(s(t), t)=0$ for all $(t_0-\eta, t_0+\eta)$. By continuity, we may assume that $F_{ss}(s, t)<0$ for every $(s,t)$ with $s \in (-\varepsilon, \varepsilon)$ and $t \in (t_0-\eta, t_0 + \eta)$, which guarantees that, for each such fixed $t$, the function $F(\cdot, t)$ restricted to the interval $s \in (-\varepsilon, \varepsilon)$ has only one critical point at $s = s(t)$, that is a strict maximum. Next, possibly after making $\eta$ even smaller, we may also assume that
	\begin{equation*}
	\max \{F(s, t);\, s \in [-1,1]\} = \max \{F(s, t);\, s \in (-\varepsilon, \varepsilon)\} = F(s(t),t)
	\end{equation*}
	for all $t \in (t_0-\eta, t_0 + \eta)$. This choice is possible thanks to the continuity of $F$ and the fact that $F(s, t_0)< F(0, t_0)$, for all $s \neq 0$. Thus, the smooth function $s(t)$ satisfies \eqref{eq2-proof1} whenever $t_0-\eta < t < t_0 +\eta$, as claimed. \\
	\indent 	In order to conclude the proof of the Lemma, define
 \begin{equation*}
		\xi(t) = F(s(t), t)-W(t) \quad \text{for all} \quad t\in (t_0-\eta,t_0+\eta).
\end{equation*} 
The function $\xi$ satisfies $\xi(t)\geq 0$ for all $t\in (t_0-\eta,t_0+\eta)$, because $\{\Sigma_s\}$ is a sweep-out of $(S^3, g(t))$ and
\begin{equation*}
	\max_{s\in[-1,1]}area(\Sigma_s,g(t)) = F(s(t),t) \quad \text{for all} \quad t\in (t_0-\eta,t_0+\eta),
\end{equation*}
by \eqref{eq2-proof1}. Moreover, $\xi$ is differentiable at $t_0$, where $\xi(t_0) = area(\Sigma, g(t_0)) - W(t_0)=0$. Hence,
	\begin{equation*}
	\xi^{\prime}(t_0) = 0 \quad \Rightarrow \quad W^{\prime}(t_0)= \nabla F(s(t_0), t_0)\cdot (s^{\prime}(t_0), 1)
	\end{equation*}
	and the formula for $W'(t_0)$ now follows from \eqref{eq1-proof1}.
\end{proof}

\indent The next lemma is the key ingredient of the proof of Theorem \ref{thm-roundsphere}, Proposition \ref{prop-star-equidistribution} and Proposition \ref{prop-necessary-Yamabe}.

\begin{lemm} \label{lemm-star-necessary}
	Let $g(t)$, $t\in [0,\epsilon)$, be a smooth family of Riemannian metrics on $S^3$ that satisfy property $(\star)$. If 
	\begin{equation*}
		W(S^3,g(0)) \geq W(S^3,g(t)) \quad \text{for every} \quad t\in [0,\epsilon),
	\end{equation*}
	then there exists an embedded index one minimal two-sphere $\Sigma$ in $(S^3,g(0))$ such that
	\begin{equation*}
		area(\Sigma,g(0)) = W(S^3,g(0)) \quad \text{and} \quad 		\int_{\Sigma}\text{Tr}_{(\Sigma,g(0))}(\partial_{t}g(0))\, dA_{g(0)} \leq 0.
	\end{equation*}
\end{lemm}
\begin{proof}
	By the first part of Lemma \ref{lemm-star-derivative-width}, the function
	\begin{equation*}
		W(t) = W(S^3,g(t)), \quad t\in [0,\epsilon),
	\end{equation*}
	is Lipchitz continuous on every interval $[0,\eta]\subset [0,\epsilon)$. Since $W$ attains a maximum at $t=0$, we have
	\begin{equation*}
		0 \geq W(t) - W(0) = \int_{0}^{t} W'(\xi)d\xi \quad \text{for every} \quad t\in [0,\eta],
	\end{equation*}
	by the Fundamental Theorem of Calculus. It follows that there exists a sequence $t_k\in [0,\epsilon)$ converging to zero such that $W$ is differentiable at $t_k$ and $W'(t_k)\leq 0$ for all $k$. \\
	\indent By the second part of Lemma \ref{lemm-star-derivative-width}, there exists a sequence of embedded index one minimal two-spheres $\Sigma_k$ in $(S^3,g(t_k))$, with area $W(t_k)$, such that
	\begin{equation*}
	\frac{1}{2}\int_{\Sigma_k}\text{Tr}_{(\Sigma_k,g(t_k))}(\partial_{t}g(t_k))\, dA_{g(t_k)} = W'(t_k)\leq 0.
	\end{equation*}
	\indent The conclusion now follows by the Compactness Theorem \ref{thm-compactness}: since stable minimal two-spheres do not exist in $(S^3,g(0))$ with area less than $W(S^3,g(0))$ by assumption, then, up to a subsequence, the two-spheres $\Sigma_k$ converge graphically, smoothly and with multiplicity one to an embedded index one minimal two-sphere $\Sigma$ in $(S^3,g(0))$ with area $\lim_{k\rightarrow 0} W(t_k) = W(0)= W(S^3,g(0))$. Moreover, by the graphical convergence,
	\begin{equation*}
		\int_{\Sigma}\text{Tr}_{(\Sigma,g(0))}(\partial_{t}g(0))\, dA_{g(0)}= \lim_{k\rightarrow \infty} \int_{\Sigma_k}\text{Tr}_{(\Sigma_k, g(t_k))}(\partial_{t}g(t_k))\, dA_{g(t_k)} \leq 0.
	\end{equation*} 
\end{proof}

\indent In order to deal with metrics that do not satisfy property $(\star)$, we prove the following technical approximation lemma, which appears also in \cite{MarNevSon} in the context of the multi-parameter Almgren-Pitts min-max theory. \\
\indent In the statement below, $\Gamma_q$ denotes the set of Riemannian metrics of class $C^{q}$ on $S^3$, $q\geq 2$.

\begin{lemm}\label{lemm-approximate-derivative-width}
Let $q \geq 4$ be an integer, and $g : [0,1] \rightarrow \Gamma_q$ be a smooth embedding. Then, there exist smooth embeddings $h : [0,1]\rightarrow \Gamma_q$, which are arbitrarily close to $g$ in the smooth topology, and subsets $J \subset [0,1]$ such that the following properties hold:
\begin{itemize}
\item[$i)$] The set $J$ has full Lebesgue measure on $[0,1]$;
\item[$ii)$] the function $W(S^3,h(t))$ is differentiable at every $\tau\in J$; and
\item[$iii)$] for each $\tau \in J$, there exist a finite collection $\{\Sigma_1, \ldots, \Sigma_N\}$ of disjoint embedded minimal two-spheres of class $C^q$ in $(S^3, h(\tau))$, and a collection $\{n_1, \ldots, n_N\}$ of positive integers, such that
\begin{align*}
 W(S^3, h(\tau)) = \sum_{k=1}^N n_k\cdot area(\Sigma_k, h(\tau)), \quad  \sum_{k=1}^N ind_{h(\tau)}(\Sigma_k) \leq 1,\\
 \quad \text{and} \quad \frac{d}{dt}\bigg|_{t=\tau} W(S^3, h(t)) = \frac{1}{2}\sum_{k=1}^N n_k \int_{\Sigma_k} \text{Tr}_{(\Sigma_k, h(\tau))}(\partial_t h(\tau)) dA_{h(\tau)}. 
\end{align*}
\end{itemize}
\end{lemm}

\begin{proof}
	(\textit{Cf}. \cite{MarNevSon}, Lemma 2). According to White \cite{Whi}, the set of all pairs $(\gamma,[u])$ consisting of a $C^q$ Riemannian metric $g$ on $S^3$ and a (equivalence class of a) minimal embedding $u : S^2 \rightarrow S^3$ of class $C^{2,\alpha}$ with respect to the metric $\gamma$ form a separable $C^{q-2}$ Banach manifold $\mathcal{M}_q$. We remark that the regularity theory says that such embeddings are actually as regular as the metric $\gamma$; in particular, they are of class $C^q$. Moreover, the projection $\Pi: (\gamma,[u])\in \mathcal{M}_q \mapsto \gamma \in \Gamma_q$ is a $C^{q-2}$ Fredholm map with Fredholm index zero, and the set of regular values is precisely the set of metrics $\gamma$ admitting no embedded minimal two-spheres with Jacobi fields. \\
	\indent Given any smooth embedding $g : [0,1] \rightarrow \Gamma_q$, we can therefore use Smale's Transversality Theorem to perturb it to another smooth embedding $h : [0,1] \rightarrow \Gamma_q$, arbitrarily close to $g$ in the smooth topology, in such way that the subset $J$ of points $\tau\in [0,1]$ such that $h(\tau)$ is a regular value of $\Pi$ and the function $W(S^3,h(t))$ is differentiable at $\tau$ has full measure (this follows by Sard-Smale Theorem and Rademarcher Theorem, given that the width of the varying metrics $h(t)$ is a Lipschitz continuous function by the proof of Lemma \ref{lemm-star-derivative-width}). \\
	\indent In particular, given any point $\tau\in J$, and a chosen sequence $t_i \rightarrow \tau$ in $[0,1]$, we have
	\begin{equation} \label{eq1-proof2}
		\frac{d}{dt}	\bigg|_{t=\tau} W(S^3,h(t)) = \lim_{i\rightarrow \infty} \frac{W(S^3,h(t_i))- W(S^3,h(\tau))}{t_i-\tau}.
	\end{equation}	  \\
	\indent Now we invoke Simon-Smith Min-max Theorem and the index estimates of Marques and Neves (see \cite{MarNev-index}, Theorem 1.2 and Subsection 1.3 therein): for each $i$, there exists a collection $\{\Sigma(i,1),\ldots, \Sigma(i,N_i)\}$ of mutually disjoint embedded minimal two-spheres and a collection $\{n(i,1),\ldots,n(i,N_i)\}$ of positive integers such that 
	\begin{equation} \label{eq2-proof2}
		W(S^3,h(t_i)) = \sum_{j=1}^{N_i}n(i,j)\cdot area(\Sigma(i,j),h(t_i))
	\end{equation}
	and
	\begin{equation*}
		 \sum_{j=1}^{N_i}ind_{h(t_i)}(\Sigma(i,j)) \leq 1.
	\end{equation*}
	We claim that that there exists a constant $C> 0$ such that the inequalities
\begin{equation*}
 C \leq area(\Sigma(i,j), h(t_i)) \leq 2W(S^3,h(\tau))
\end{equation*}
hold for large enough $i$ and all $j=1,\ldots,N_i$. In fact, as the metrics $h(t_m)$ are converging to $g$, the constant $C>0$ is a uniform constant coming from the monotonicity formula for minimal surfaces (see \cite{Sim}), while the upper bound follows from \eqref{eq2-proof2} and the convergence of the numbers $W(S^3,h(t_i))$ to $W(S^3,h(\tau))$. The claim implies that, for all sufficiently large $i$, the integers $N_i$ and the multiplicities $n(i,j)$, $j=1,\ldots, N_i$, are uniformly bounded. In particular, possibly after passing to appropriate subsequences (which we do not relabel), we can assume that $N_i=N$ for some fixed integer $N$, and $n(i,j)=n_j$ is independent of $i$ for every $j = 1, \ldots, N$. \\
\indent Now, recalling that $(S^3,h(\tau))$ contains no minimal two-sphere admitting Jacobi fields (by the definition of $J\ni \tau$), we apply the Compactness Theorem \ref{thm-compactness} (and the remark after its statement): up to subsequences, the collections of two-spheres $\{\Sigma(i,1),\ldots,\Sigma(i,N)\}_{i}$ and integers $\{n_1,\ldots,n_N\}$ are such that each sequence $\{\Sigma(i,j)\}_i$ converges smoothly and graphically with multiplicity one to an embedded minimal two-sphere $\Sigma_j$ in $(S^3,h(\tau))$, where the collection $\{\Sigma_1,\ldots,\Sigma_N\}$ consists of surfaces satisfying
	\begin{equation} \label{eq3-proof2}
		W(S^3,h(\tau)) = \sum_{j=1}^{N}n_j\cdot area(\Sigma_j,h(\tau))
	\end{equation}
	and
	\begin{equation*}
		\sum_{j=1}^{N}ind_{h(\tau)}(\Sigma_j) \leq 1.
	\end{equation*}
	Moreover, for each $i=1,\ldots,N$ we can compute
	\begin{equation} \label{eq4-proof2}
		\lim_{i\rightarrow \infty} \frac{area(\Sigma(i,j),h(t_i)) - area(\Sigma_j,h(\tau))}{t_i-\tau} = \frac{1}{2} \int_{\Sigma_j} \text{Tr}_{(\Sigma_j,h(\tau))}(\partial_t h(\tau)) dA_{h(\tau)}.
	\end{equation}
	Indeed, notice that the pair $(h(\tau),\Sigma_j)$ belongs to a one-di\-men\-sio\-nal submanifold in $\mathcal{M}_q$ containing the all the pairs $(h(t_m),\Sigma(i,j))$ eventually, by virtue of the fact that $\tau\in J$ is a regular value of the projection $\Pi$ and the graphical convergence of $\{\Sigma(i,j)\}_i$ to $\Sigma_i$. Thus, in equation \eqref{eq4-proof2}, we have just calculated the derivative the area functional at $t=\tau$ along this curve in the space $\mathcal{M}_q$. \\
	\indent The final remark is that any two surfaces $\Sigma_j$ and $\Sigma_k$ in the collection $\{\Sigma_1,\ldots,\Sigma_j\}$ are either disjoint or equal. Indeed, if $\Sigma_j$ intersects $\Sigma_k$, the intersection contain infinitely many points by the Maximum Principle. Therefore, there exists a point $p \in \Sigma_j\cap \Sigma_k$ around which both convergences $\Sigma(i,j)\rightarrow \Sigma_j$ and $\Sigma(i,k)\rightarrow \Sigma_k$ are graphical. Since $\Sigma(i,j)$ and $\Sigma(i,k)$ are disjoint for all $i$, it follows that $\Sigma_k$ is in one of the sides of $\Sigma_j$ near $p$. By the Maximum Principle, these two surfaces must coincide. Thus, adding the multiplicities of equal surfaces, and renaming the two-spheres $\Sigma_j$ if necessary, item $iii)$ and the formula for the derivative of $W(S^3,h(\tau))$ at $t=\tau$ follows from \eqref{eq1-proof2} substituting the values of the widths given by \eqref{eq2-proof2} and \eqref{eq3-proof2} and using the formula \eqref{eq4-proof2} for the derivatives of the area.
\end{proof}

\section{The width of the round three-sphere} \label{Sec-3}

\subsection{Review on the Yamabe Flow} Introduced by R. Hamilton as the backwards gradient flow of the \textit{normalised total curvature (Hilbert-Einstein) functional},
\begin{equation*}
	E(g) = \frac{\int_{M} R_g dV_g}{vol(M,g)^{(n-2)/n}},
\end{equation*}
restricted to a conformal class of metrics on a closed manifold $M^n$, $n\geq 3$, the normalised Yamabe flow is given by 
\begin{equation} \label{eq-normalised-YF}
	\frac{\partial }{\partial t} g = (r_g-R_{g}) \, g,
\end{equation} 
where $R_g$ denotes the scalar curvature of the metric $g$ and 
\begin{equation*}
 r_g = \frac{1}{vol(M,g)}\int_{M} R_g dV_g
\end{equation*}
denotes its average value over $(M,g)$. \\ 
\indent Since the functional $E$ is invariant under scaling, the normalised Yamabe flow keeps the volume constant. Moreover, the functional $E$ is monotonically non-increasing as the time goes on (as it happens in all backwards gradient flows). In particular, $r_{g(0)} \geq r_{g(t)}$ for all $t\geq 0$ while the flow exists. \\
\indent One of the first papers on the Yamabe flow was written by Chow \cite{Cho}, who restricted his attention to conformally flat metrics, \textit{i.e} Riemannian metrics in the conformal class of the standard round metric of the $n$-sphere. The theory developed greatly since then; we refer the reader to the recent survey \cite{BreMar}, Section 6, for further information. The results needed for the purposes of this paper were already contained in \cite{Cho}, and we summarise them in the following
\begin{prop} (\textit{cf}. Theorem 1.2 and Lemma 2.9 in \cite{Cho})\label{prop-Yamabe} \\
	\indent Let $g_0$ be a conformally flat Riemannian metric on $S^3$ with positive Ricci curvature. Then 
	\begin{itemize}
		\item[$i)$] the normalised Yamabe flow \eqref{eq-normalised-YF} starting at $g(0)=g_0$ has a unique solution $g(t)$ that exists for all $t\geq 0$; 
		\item[$ii)$] for every $t\geq 0$, the metric $g(t)$ has positive Ricci curvature; and
		\item[$iii)$] the flow converges smoothly to a metric $g_{\infty}$ on $S^3$ with constant  sectional curvature.
	\end{itemize}
\end{prop}  

\subsection{Proof of Theorem \ref{thm-roundsphere}} Let $g$ be a Riemannian metric on $S^3$ that is conformally flat and has positive Ricci curvature. Without loss of generality, we will assume that $(S^3,g)$ has the same volume as a unit three-sphere in the Euclidean four-dimensional space. \\
\indent Let $g(t)$, $t\geq 0$, be the Yamabe flow \eqref{eq-normalised-YF} starting at the metric $g(0)=g$, as given by Proposition \ref{prop-Yamabe}. The metrics $g(t)$ converges to a metric $g_{\infty}$ with constant sectional curvature $K$ on $S^3$. Given the chosen volume normalisation at $t=0$, the fact that the flow does not change the volume forces $K$ to be equal to one. Thus, the widths 
	\begin{equation*}
		W(t) = W(S^3,g(t)), \quad t\in [0,\infty),
	\end{equation*}
	vary continuously (see Lemma \ref{lemm-star-derivative-width}) and converge to $W(S^3,g_\infty)=4\pi$ as $t$ goes to infinity, while the average value of the scalar curvature $r_{g(t)}$ is always greater than or equal to $r_{g_{\infty}}=6$.\\
	\indent Suppose, by contradiction, that $W(0) > 4\pi$. By continuity, there exists $\tau\in [0,\infty)$ such that the $W(\tau)\geq W(t)$ for all $t\in [0,\infty)$. Since all metrics $g(t)$ satisfy property $(\star)$ by virtue of having positive Ricci curvature, we can use Lemma \ref{lemm-star-necessary} at time $t=\tau$ to conclude that there exists an embedded index one minimal two-sphere $\Sigma$ in $(S^3,g(\tau))$ with area $W(\tau)$ and such that
	\begin{multline} \label{eq-derivative-width-normalised-YF}
		0\geq \frac{1}{2}\int_{\Sigma} \text{Tr}_{(\Sigma,g(\tau))}(\partial_t g(\tau))dA_{g(\tau)} \\ =  \int_{\Sigma}(r_{g(\tau)} - R_{g(\tau)})dA_{g(\tau)} = W(\tau)r_{g(\tau)} - \int_{\Sigma} R_{g(\tau)} dA_{g(\tau)}.
	\end{multline} 
	Using the Gauss equation for minimal surfaces
	$$2K = R - 2Ric(N,N) - |A|^2,$$ 
	the well-known estimate   
	\begin{equation*}
		\int_{\Sigma} Ric(N,N) + |A|^2 dA_{g(\tau)}\leq 8\pi
	\end{equation*}
	for index one minimal two-spheres (see \cite{MarNev-Duke}, Appendix A), and the Gauss-Bonnet theorem, we conclude that
	\begin{equation*}
		\int_{\Sigma} R_{g(\tau)}dA_{g(\tau)} \leq 24\pi.
	\end{equation*}
	Combining the above inequality with \eqref{eq-derivative-width-normalised-YF}, we have
	\begin{equation*}
		 W(\tau)r_{g(\tau)} \leq  24\pi
	\end{equation*}
	But then $r_{g(\tau)}\geq r_{g_{\infty}}=6$ implies $W(\tau) \leq 4\pi$. This gives a contradiction with the fact that, as an absolute maximum of $W(t)$, the point $\tau$ is such that $W(\tau) \geq W(0) > 4\pi$. \\
	\indent Therefore $W(S^3,g)\leq 4\pi$, as claimed. In the case of equality, we consider again the Yamabe flow $g(t)$ starting at $g(0)=g$. Since all metrics $g(t)$ have the same volume, are conformally flat and have positive Ricci curvature, the previous analysis yields $W(S^3,g(t)) \leq 4\pi$ for all $t\geq 0$. In particular, $t=0$ is a point of maximum of the function $W(t)$. The same argument above allows us to conclude that $W(S^3,g)r_{g} \leq 24\pi$. But the flow started at $g$ with $W(S^3,g)=4\pi$, and $r_{g} \geq r_{\infty}=6$. These three inequalities imply that the metrics  $g$ and $g_{\infty}$ have the same average scalar curvature. As they also have the same volume, Obata's Theorem \cite{Oba} now finishes the proof, for we have shown that $g$ attains the minimum possible value of the Yamabe functional in the conformal class $[g_{\infty}]$ and therefore must have constant sectional curvature.

\section{The width of homogeneous three-spheres} \label{Sec-4}

\subsection{A general, sharp estimate} Every homogeneous $(S^3,g)$ can be identified with the Lie group $SU(2)$ endowed with a left-invariant metric $g$ (\textit{cf}. Theorem 2.4 in \cite{MeePer}, for example). The right-invariant vector fields form a three-dimensional vector subspace of Killing fields of $(SU(2),g)$, because their flow is by left translations. These vector fields either vanish identically, or vanish nowhere. Using these observations, one can prove the following

\begin{lemm} \label{lemm-no-stable}
	A homogeneous Riemannian three-sphere contains no embedded stable minimal two-spheres.
\end{lemm}

\begin{proof}
	(\textit{Cf}. \cite{MeeMirPerRos}, Section 4). Given a minimal two-sphere $\Sigma$ in $(S^3,g)$, oriented by the unit normal vector field $N$, the map that assigns to each right-invariant vector field $X$ on $S^3$ the Jacobi function $g(X,N)$ on $\Sigma$ is linear and has trivial kernel, because otherwise the two-sphere $\Sigma$ would admit a nowhere vanishing tangent vector field. In particular, $\Sigma$ admits at least three linearly independent Jacobi function. Hence, $\Sigma$ cannot be stable. 
\end{proof}

\begin{thm} \label{thm-homogeneous-scalar-curvature}
	If $g$ is a homogeneous Riemannian metric on the three-sphere with constant scalar curvature 6, then
		\begin{equation*}
			W(S^3,g) \leq 4\pi,		
		\end{equation*}		 	
		and equality holds if and only if $(S^3,g)$ has constant sectional curvature one.
\end{thm}

\begin{proof}
	By uniqueness of the Ricci flow given the initial condition, and the fact that the flow equation is invariant under isometries, the maximally extended Ricci flow $g(t)$, $t\in [0,T)$, starting at $g(0)=g$ is a flow by homogeneous metrics. Lemma \ref{lemm-no-stable} shows that any homogeneous $(S^3,g)$ satisfies the condition $(\star)_0$ of Marques and Neves \cite{MarNev-Duke}. Hence, Proposition 4.2 in \cite{MarNev-Duke} applies for the Ricci flow $g(t)$ and controls how slowly the width can possibly converge to zero. On the other hand, since $g(0)$ has positive scalar curvature, the maximum principle controls how fast the (constant) scalar curvature of $g(t)$ can possibly grow up to infinity, and the result follows by precisely the same argument presented in the proof of Theorem 4.4 in \cite{MarNev-Duke}. 
\end{proof}

\indent Given the identification of homogeneous metrics on $S^3$ with left-invariant metrics on $SU(2)$, one can show that these metrics are parametrised by three positive real numbers (see \cite{Mil}, Section 4). The above theorem provides an upper bound for the normalised width of a homogeneous three-sphere of the form $24\pi/RV^{2/3}$, where $R$ denotes its scalar curvature and $V$ denotes its volume. Since both quantities can be explicitly computed as a function of the parameters (\textit{e.g.} in \cite{Mil}, page 306), we obtained an upper bound for the normalised width that is explicitly computable in terms of these parameters.

\subsection{The width of Berger spheres}

\indent The Berger metric $g_{\rho}$, $\rho>0$, is the homogeneous metric on $S^3$ obtained by deforming the standard round metric so that the norm of the vector field generating the standard Hopf action becomes $\rho$, and nothing else changes. Notice that $g_1$ is the round metric. The geometry of $(S^3,g_{\rho})$ is very well understood. For example, its Ricci curvature is positive if and only if $0<\rho<\sqrt{2}$, and its scalar curvature is given by $8-2\rho^2$. Furthermore, its minimal two-spheres have been classified by Abresch and Rosenberg \cite{AbrRos}: up to ambient isometries, there exists only one minimal two-sphere in $(S^3,g_{\rho})$, which is embedded and has index one (see \cite{TorUrb} for more details). Thus, the area of this unique minimal two-sphere is equal to the width of $(S^3,g_{\rho})$.  \\
\indent The classification result allows us to compute the widths of the Berger metrics rather explicitly (at least numerically). The following formula was computed in \cite{AmbMon}, Section 5: for all $\rho > 0$, 
\begin{equation*}
	\frac{W(S^3,g_{\rho})}{vol(S^3,g_{\rho})^{\frac{2}{3}}} = \sqrt[\leftroot{-1}\uproot{2}\scriptstyle 3]{\frac{2}{\pi}}\int_{0}^{\pi}\sin(s)\sqrt{(1-\sin^{2}(s))\rho^{-\frac{4}{3}} + \sin^{2}(s)\rho^{\frac{2}{3}}}ds.
\end{equation*} 
\noindent It is now a straightforward exercise to reach the conclusions stated in Proposition \ref{thm-bergermetrics}.

\subsection{The integral-geometric formula} Extending Abresch and Rosenberg classification of minimal two-spheres in Berger spheres \cite{AbrRos}, Meeks, Mira, P\'erez and Ros \cite{MeeMirPerRos} classified minimal two-spheres in any homogeneous $(S^3,g)$. In particular, they showed that the minimal two-spheres in these spaces are unique up to ambient isometries. Moreover, they are all embedded, have index one, and their area is equal to $W(S^3,g)$. \\
\indent In \cite{AmbMon}, we proved the following integral-geometric formula for all homogeneous Riemannian metrics $g$ on the three-sphere:
\begin{equation} \label{eq-homogeneous-equidistribution}
		\int_{\mathcal{G}^{+}} \left(\frac{1}{W(S^3,g)}\int_{\Sigma} f dA_{g} \right) d\mathcal{G}^{+}_{g} = \int_{S^3} f dV_{g} \quad \text{for all} \quad f\in C^{0}(S^3).
\end{equation}
Here, $\mathcal{G}^{+}$ denotes the set of all minimal two-spheres in $(S^3,g)$, which can be shown to be in bijection with $S^3$ itself thanks of the classification theorem. We refer the reader to \cite{AmbMon}, Theorems 2.1 and 3.1, for a more detailed explanation of the meaning of formula \eqref{eq-homogeneous-equidistribution}, and its proof. \\
\indent Transforming the integral over $\mathcal{G}^{+}$ in \eqref{eq-homogeneous-equidistribution} into a limit of Riemann sums, it is immediate to see that the volume element of a homogeneous three-sphere and the Radon measures associated to each element of the compact set $\mathcal{G}^+$ satisfy one of the equivalent conditions expressed in Theorem \ref{thm-abstract-equidistribution}, namely condition $iii)$. In view of condition $iv)$ in the same Theorem, we have thus verified that the conclusion of Proposition \ref{prop-star-equidistribution} holds for every homogeneous three-spheres, as claimed.

\section{Locally maximising metrics} \label{Sec-5}

\subsection{On equidistribution: first proof} We prove Proposition \ref{prop-necessary-equidistribution} under the extra assumption that the locally maximising metric satisfies property $(\star)$. Before starting the proof, we make a simple observation: 
\begin{lemm} \label{lemm-constant-volume-variations}
	Let $g$ be a Riemannian metric on $S^3$. For every smooth function $f$ on $S^3$ with zero average, there exists a smooth family of Riemannian metrics $g(t)$, $t\in[0,\epsilon)$, such that
	\begin{itemize}
		\item[$i)$] $g(0)=g$ and $\partial_t g(0)=fg$;
		\item[$ii)$] all metrics $g(t)$ belong to the same conformal class of $g$; and
		\item[$iii)$] $vol(S^3, g(t))=vol(S^3,g)$ for every $t\in [0,\epsilon)$.
	\end{itemize}	 
\end{lemm}
\begin{proof}	
	For $\epsilon>0$ sufficiently small, consider the smooth path starting at the metric $g$ defined by
	\begin{equation*}
		g(t)=\frac{vol(S^3,g)^{\frac{2}{3}}(1+ft)}{vol(S^{3},(1+ft)g)^{\frac{2}{3}}}g \quad \text{for all} \quad t \in [0,\epsilon).
	\end{equation*}
 	If $\epsilon$ is sufficiently small, each $g(t)$ is a Riemannian metric in the conformal class of $g$ and with the same volume as $g$. Moreover, $g(0)=g$ and
 	\begin{equation*}
 		\frac{d}{dt}\Bigr|_{t=0} vol(S^3,(1+tf)g) = \frac{3}{2}\int_{S^3}fdV_g =0 \quad \Rightarrow \quad \partial_t g(0) = f g.
 	\end{equation*}
\end{proof}
	
\begin{prop} \label{prop-star-equidistribution}
	Let $g$ be a Riemannian metric on the three-sphere that is a local maximum of the normalised width inside its conformal class. Assume that $(S^3,g)$ satisfies property $(\star)$. Then there exists a sequence $\{\Sigma_{i}\}$ of embedded minimal two-spheres in $(S^3,g)$, with Morse index equal to one and area equal to $W(S^3,g)$, such that
	\begin{equation*}
		\lim_{k\rightarrow +\infty} \frac{1}{k} \sum_{i=1}^{k}\int_{\Sigma_{i}} fdA_g = \frac{W(S^3,g)}{vol(S^3,g)}\int_{S^3} f dV_{g} \quad \text{for all}\quad f\in C^{0}(S^3).
\end{equation*}	 
\end{prop}

\begin{proof}
	Given an arbitrary smooth function $f$ on $(S^3,g)$ with zero average, let $g(t)$, $t\in [0,\epsilon)$, be a smooth family of metrics starting at $g(0)=g$ as described in Lemma \ref{lemm-constant-volume-variations}. Taking a smaller $\epsilon>0$, if necessary, we may assume that all metrics $g(t)$ satisfy property $(\star)$ as well, by Lemma \ref{lemm-openness}. Since all metrics of the family have the same volume as $g$ and lie in the same conformal class, the assumption on $g$ allows us to apply Lemma \ref{lemm-star-necessary}: there exists an embedded index one minimal two-sphere $\Sigma$ such that 
	\begin{equation*}
		area(\Sigma,g)=W(S^3,g) \quad \text{and} \quad \int_{\Sigma} f dA_g = \frac{1}{2}\int_{\Sigma} tr_{(\Sigma,g)}(\partial_tg(0))dA_g \leq 0.	
	\end{equation*}	 
\indent Approximating continuous functions by smooth functions in the $C^0$ norm, and using again the Compactness Theorem \ref{thm-compactness}, it is straightforward to check that the last assertion holds not only for smooth, but all continuous functions $f$ on $S^3$ with zero average. \\
\indent Thus, we can apply the abstract Theorem \ref{thm-abstract-equidistribution} in the Appendix, considering the volume element $dV_g$ as a Radon measure on $S^3$ and $\mathcal{Y}$ as the non-empty set of Radon measures originating from integration over an embedded index one minimal two-sphere in $(S^3,g)$ that has area equal to $W(S^3,g)$. The result follows: see items $ii)$ and $iv)$ of Theorem \ref{thm-abstract-equidistribution}.	
\end{proof}

\subsection{On equidistribution: second proof} The next lemma, which should be compared to Lemma \ref{lemm-star-necessary}, is the key for the proof of Proposition \ref{prop-necessary-equidistribution}.

\begin{lemm} \label{lemm-main-general}
Let $g$ be a Riemannian metric on $S^3$ that is a local maximum of the normalised width in its conformal class. For every continuous function $f$ on $(S^3,g)$ satisfying
\begin{equation*}
	\int_{S^3} fdV_g < 0,
\end{equation*} 
there exists positive integers $n_1,\ldots, n_N$, and disjoint embedded minimal two-spheres $\Sigma_{1},\ldots, \Sigma_{N}$ in $(S^3,g)$ such that
\begin{equation*}
W(S^3,g) = \sum_{j=1}^{N} n_j\cdot area(\Sigma_j,g),  \quad \sum_{j=1}^{N} ind_{g}(\Sigma_j) \leq 1
\end{equation*}
\begin{equation*}
\text{and} \quad \sum_{j=1}^N n_j\int_{\Sigma_j} f dA_{g} \leq 0.
\end{equation*}
\end{lemm}
\begin{proof}
First of all, let us consider the case where $f$ is a smooth function on $S^3$ with $\int_{S^3}fdV_g < 0$. Define the path of metrics
\begin{equation*}
g(t) = \left(1+\frac{3}{2}tf\right)^{\frac{2}{3}}g \quad \text{for} \quad t\in [0,T],
\end{equation*}
where $T$ is small enough to ensure that the conformal factor is positive. Then $g(0)=g$, 
\begin{equation} \label{eq-compare-volume}
vol(S^3, g(t)) = vol(S^3, g(0)) + \frac{3}{2} t\int_{S^3} fdV_{g(0)}< vol(S^3, g(0))
\end{equation} 
for all $t \in (0, T]$, and
\begin{equation*}
\partial_t g(t) = \left(1+\frac{3}{2}tf\right)^{-\frac{1}{3}} fg \quad \Rightarrow \quad \partial_t g(0) = fg.
\end{equation*}
Since $g(t)\in [g(0)]$ for all $t \in (0, T]$, the maximality hypothesis on $g=g(0)$ and \eqref{eq-compare-volume} imply that there exists $\delta\in (0,T]$ such that
\begin{equation*}
W(S^3,g(t)) \leq W(S^3,g(0)) \bigg(\frac{vol(S^3, g(t))}{vol(S^3, g(0))}\bigg)^{\frac{2}{3}}< W(S^3,g(0))
\end{equation*} 
for all $t \in (0, \delta]$. \\
\indent Fix an integer $q\geq 4$. Observe that $g : [0, \delta]\rightarrow \Gamma_q$ is a smooth embedding, where $\Gamma_q$ represents the space of Riemannian metrics of class $C^q$ on $S^3$ (possibly after decreasing $\delta$). For each large enough positive integer $i$, define $\delta_i = 1/i < \delta$ and apply Lemma \ref{lemm-approximate-derivative-width} to the restriction $g : [0, \delta_i]\rightarrow \Gamma_q$ to obtain a sequence of smooth embeddings
\begin{equation*}
h_i : [0, \delta_i] \rightarrow \Gamma_q
\end{equation*}
that are as close as we wish to $g : [0, \delta_i]\rightarrow \Gamma_q$ in the smooth topology, and subsets $J_i \subset [0, \delta_i]$ of full Lebesgue measure satisfying the properties $i)$, $ii)$ and $iii)$ listed in the statement of that Lemma, plus the further inequality
\begin{equation}\label{proof3-eq1}
W(S^3,h_i(\delta_i))< W(S^3,h_i(0)).
\end{equation}
This last inequality can be guaranteed because the widths vary continuously and $W(S^3,g(\delta_i))< W(S^3,g(0))$. \\
\indent Since the function $W(S^3,h_i(t))$, $t\in [0,\delta_i]$, is Lipschitz continuous (see the proof of Lemma \ref{lemm-star-derivative-width}), inequality \eqref{proof3-eq1} and the Fundamental Theorem of Calculus imply that there exists $\tau_i \in J_i \subset [0, \delta_i]$ such that
\begin{equation} \label{proof3-eq2}
\frac{d}{dt}\bigg|_{t=\tau_i} W(S^3,h_i(t))\leq 0.
\end{equation}
\indent For each $i$, let $\{\Sigma(i,j)\}$ and $\{n(i,j)\}$, $j=1, \ldots, N_i$, be the sequences given by item $iii)$ of the statement of Lemma \ref{lemm-approximate-derivative-width}, associated with $h_i$ and $\tau_i \in J_i$. As explained in the proof the same Lemma, possibly after passing to subsequences, we can assume that $N_i=N$ for some fixed integer $N$, and $n(i,j)=n_j$ is independent of $i$, for every $j = 1, \ldots, N$. Moreover, up to a reordering of the hypersurfaces $\Sigma(i,j)$ on the index $j$, we may say that the Morse index of $\Sigma(i,j)$ in $(S^3,h_i(\tau_i))$ equals one for $j=1$ and zero otherwise. This is possible because 
\begin{equation*}
\sum_{j=1}^N ind_{h_i(\tau_i)}(\Sigma(i,j)) \leq 1.
\end{equation*}
\indent Observe that the metrics $h_i(\tau_i)$ converge to $g=g(0)$ in the $C^q$ topology, as well as $\partial_t{h}_i(\tau_i)$ converges to $\partial_t{g}(0)=fg$ in the same topology. This follows because $h_i : [0, \delta_i] \rightarrow \Gamma_q$ can be taken arbitrarily close to $g: [0, \delta_i]\rightarrow \Gamma_q$, and the numbers $\tau_i \in [0, \delta_i]$ are decreasing to zero. \\
\indent We now apply the Compactness Theorem \ref{thm-compactness} for each sequence $\{\Sigma(i,j)\}_i$ to obtain a list $\{\Sigma_1, \ldots, \Sigma_N\}$ of connected (smooth) embedded minimal two-spheres in $(S^3, g)$ such that, up to subsequences, each $\{\Sigma(i,j)\}_i$ converges to $\Sigma_j$ in the following sense: $\Sigma(i,j)\rightarrow m_j\Sigma_j$ as varifolds for some $m_j\in \{1,2\}$, and locally graphically with $m_j$-sheets in the $C^{q-1}$ topology except possibly at a single point of $\Sigma_1$ if it happens that $m_1=2$ (notice that $m_j=1$ for all $j>1$, as all $\Sigma(i,j)$ are stable minimal two-spheres). Also, $ind_{g}(\Sigma_1)\leq 1$ and $ind_{g}(\Sigma_j)=0$ for $j\neq 1$.  Furthermore, as explained at the end of the proof of Lemma \ref{lemm-main-general}, the Maximum Principle implies that, if $\Sigma_j \cap \Sigma_k \neq \varnothing$, then the surfaces coincide.\\
\indent In summary, the minimal two-spheres $\Sigma_1,\ldots,\Sigma_N$, and their respective multiplicities $\overline{n}_1,\ldots,\overline{n}_N$, are such that sum of their Morse indices is at most one, and 
\begin{multline*}
W(S^3,g) = \lim_{i\rightarrow \infty} W(S^3,h_i(\tau_i))) 
    \\ = \lim_{i\rightarrow \infty}\sum_{j=1}^N n_j\cdot area(\Sigma(i,j), h_{i}(\tau_i)) = \sum_{j=1}^N \overline{n}_j\cdot area(\Sigma_j, g).
\end{multline*}
(Remark that, by the above discussion, $\overline{n}_{j}=m_jn_j$, where $m_j=1$ for all $j$ except possibly for $j=1$ if $\Sigma(i,1)$ converges to $\Sigma_1$ with multiplicity $m_1=2$). Moreover, by the construction of $\tau_i$, $\Sigma(i,j)$, and $n_j$, we also have
\begin{align*}
\sum_{j=1}^N \frac{n_j}{2} \int_{\Sigma(i,j)} \text{Tr}_{(\Sigma(i,j), h_i(\tau_i))}(\partial_t h_i(\tau_i)) dA_{h_i(\tau_i)} = \frac{d}{d t}\bigg|_{t=\tau_i} W(S^3,h_i(\tau_i)) \leq 0
\end{align*}
\noindent for every $i$, because of \eqref{proof3-eq2}. Since $\partial_{t}h_i(\tau_i)$ converges continuously to $\partial_t g(0)=fg$, and $\Sigma(i,j)$ converges to $\Sigma_j$ as described above, we can pass the inequalities above to the limit to obtain 
\begin{equation*}
	\sum_{j=1}^{N} \overline{n}_j\int_{\Sigma_j} f dA_{g}  = \frac{1}{2}\sum_{j=1}^{N} \overline{n}_j\int_{\Sigma_j} \text{Tr}_{(\Sigma_j, g)} (\partial_t g(0)) dA_{g} \leq 0
\end{equation*}
and conclude that the collection of two-spheres $\{\Sigma_1,\ldots,\Sigma_N\}$ and the positive numbers $\overline{n}_1,\ldots,\overline{n}_N$ have all the desired properties (to be precise, we might also need to rename the surfaces that coincide and add up their multiplicities, to guarantee disjointness of the collection $\{\Sigma_1,\ldots,\Sigma_N\}$). \\
\indent  This finishes the proof in the case of smooth functions $f$. The statement for continuous functions follows by approximating them by smooth functions in the $C^0$ norm, applying the result for smooth functions, and then using the same compactness and sequence-picking arguments as above.
\end{proof}

\begin{prop} \label{prop-equidistribution}
	If $g$ is a Riemannian metric on the three-sphere that is a local maximum of the normalised width inside its conformal class, then there exists a sequence $\{\Sigma_{i}\}$ of embedded minimal two-spheres in $(S^3,g)$, with Morse index at most one and area at most $W(S^3,g)$, such that, for every continuous function $f$ on $S^3$,
	\begin{equation} \label{eq-particular-equidistribution}
			\lim_{k\rightarrow +\infty} \frac{1}{\sum_{i=1}^{k}area(\Sigma_i,g)}\sum_{i=1}^{k} \int_{\Sigma_{i}} fdA_g = \frac{1}{vol(S^3,g)}\int_{S^3} f dV_{g}.
\end{equation}	 
\end{prop}
\begin{proof}
	In view of Lemma \ref{lemm-main-general}, we can apply the abstract Theorem \ref{thm-abstract-equidistribution} in the Appendix, considering the volume element $dV_g$ as a Radon measure on $S^3$ and $\mathcal{Y}$ as the non-empty set of Radon measures of the form
\begin{equation}
	\mu (f) = \sum_{j=1}^{N}n_j \int_{\Sigma_j} f dA_g,
\end{equation} 
where $n_1,\ldots,n_N$ are positive integers and $\{\Sigma_1,\ldots,\Sigma_N\}$ is a collection of disjoint embedded minimal two-spheres in $(S^3,g)$ such that
\begin{equation*}
	\sum_{j=1}^{N} ind_{g}(\Sigma_j) \leq 1 \quad \text{and} \quad \sum_{j=1}^{N} n_j\cdot area(\Sigma_j) = W(S^3,g).
\end{equation*}
Notice that each surface $\Sigma_j$ has at least some uniform amount of area (by the monotonicity formula). Thus, the numbers $N$ and $n_1,\ldots,n_{N}$ cannot be arbitrarily big, and the set $\mathcal{Y}$ is compact in the weak-$*$ topology (the argument uses the Compactness Theorem \ref{thm-compactness} and its details have already been explained during the proof of Lemma \ref{lemm-main-general}). \\
\indent	It remains only to justify the particular presentation of the equidistribution formula \eqref{eq-particular-equidistribution}, which is slightly different from the formula in item $iv)$ in Theorem \ref{thm-abstract-equidistribution} - see the remark at the end of the Appendix A.2 for further details.
\end{proof}

\subsection{On the Yamabe invariant} Let $g$ be a metric on $S^3$ that satisfies property $(\star)$. Lemma \ref{lemm-star-necessary} and Lemma \ref{lemm-constant-volume-variations} imply that every zero average smooth function can be used to test whether $(S^3,g)$ is a local maximum of the normalised width in $[g]$ or not. If the answer is positive, each zero average smooth function on $S^3$ gives rise to an integral inequality that must be satisfied by some embedded index one minimal two-sphere in $(S^3,g)$ whose area realises the width. In the result below, we use the average of the scalar curvature to deduce another necessary condition for local maxima of the normalised width (\textit{cf}. Proposition \ref{prop-intro-Yamabe}).

\begin{prop} \label{prop-necessary-Yamabe}
	Let $[g]$ be a conformal class of Riemannian metrics on the three-sphere. Assume that $g$ satisfies property $(\star)$. If $g$ is a local maximum of the normalised width in $[g]$, then 
	\begin{equation*}
		\mathcal{Y}(S^3,[g])\frac{W(S^3,g)}{vol(S^3,g)^{\frac{2}{3}}} \leq W(S^3,g)\dashint_{S^3} R_g dV_g \leq 24\pi.
	\end{equation*}
\end{prop}

\begin{proof}
	It follows from Lemmas \ref{lemm-star-necessary} and \ref{lemm-constant-volume-variations} that, given the zero average smooth function
	\begin{equation*}
		\dashint_{S^3}R_{g} dV_g - R_g
	\end{equation*}		
	on $(S^3,g)$, there exists an embedded index one minimal two-sphere $\Sigma$ in $(S^3,g)$ such that
	\begin{equation*}
		area(\Sigma,g)=W(S^3,g) \quad \text{and} \quad \int_{\Sigma} \left(\dashint_{S^3} R_{g}dV_g - R_{g}\right) dV_g \leq 0.
	\end{equation*}
By the same estimates used in the proof of Theorem \ref{thm-roundsphere} in Section 3.2, we conclude that 
	\begin{equation*}
		\frac{W(S^3,g)}{vol(S^3,g)^{\frac{2}{3}}} \left(\frac{\int_{S^3} R_g dV_g}{vol(S^3,g)^{\frac{1}{3}}} \right)= W(S^3,g)\dashint_{S^3} R_g dV_g \leq 24\pi
	\end{equation*}
By the definition of $\mathcal{Y}(S^3,[g])$, the proof is now finished.
\end{proof}

\indent Theorem \ref{thm-homogeneous-scalar-curvature} shows that homogeneous metrics with positive scalar curvature on the three-sphere satisfy the second inequality of the above proposition. Homogeneous metrics are stable critical points of the Hilbert-Einstein functional in their respective conformal classes (see \cite{BetPic} and \cite{Lau}); we do not know if non-round homogeneous metrics with positive scalar curvature are absolute minima.

\section*{Appendix}
\renewcommand{\thesubsection}{\Alph{subsection}}
\setcounter{subsection}{0}

\subsection{A compactness theorem for minimal two-spheres} The compactness result for a sequence of embedded minimal two-spheres of Morse index at most one and bounded area, which is used in this paper repeatedly, is described below. The statement is essentially a rephrasing of Sharp's Compactness Theorem for varying metrics (see in particular Theorem A.6 in \cite{Sha} and the comments on regularity theory in Remark A.4). It becomes slightly simpler because $S^3$ contains no embedded one-sided surfaces, a condition that rules out certain possible degenerations. Moreover, we can extract further information about the multiplicity from the assumption on the index, and about the limit from the assumption that all surfaces $\Sigma_k$ in the sequence are two-spheres: the multiplicity must be at most two and the limit $\Sigma$ must be necessarily a two-sphere as well, for under the convergence only a single catenoidal neck is possibly lost (see \cite{AmbBuzCarSha}, Theorem 4, and further discussions therein).

\begin{thmA}\label{thm-compactness}
(\textit{Cf.} \cite{Sha}, Theorem A.6 and \cite{AmbBuzCarSha}, Theorem 4). \\
	\indent Let $\{g_k\}$ be a sequence of Riemannian metrics of class $C^{q}$ on $S^3$, $q\geq 3$, and $\{\Sigma_k\}$ a sequence of embedded minimal two-spheres of class $C^q$  in $(S^3,g_k)$ for each integer $k$. \\
	\indent Assume that $g_k$ converges to a Riemannian metric $g$ in the $C^{q}$ topology, and that there exists a constant $C> 0$ such that $area(\Sigma_k,g_k) \leq C$ for every integer $k$. \\
	\indent If the two-spheres $\Sigma_k$ have Morse index at most one in $(S^3,g_k)$ for every $k$, then, after passing to a subsequence, there exists an embedded minimal two-sphere $\Sigma$ of class $C^q$ in $(S^3,g)$ and $m\in\{1,2\}$ such that 
	\begin{itemize}
		\item[$i)$] (Varifold convergence). For every continuous function $f$ on $S^3$,
		\begin{equation*}
			\lim_{k\rightarrow \infty}\int_{\Sigma_k} f dA_{g_k} = m\int_{\Sigma} fdA_g.
		\end{equation*}
		In particular,
		\begin{equation*}
			m\cdot area(\Sigma,g) = \lim_{k\rightarrow \infty} area(\Sigma_k,g_k) \leq C.
		\end{equation*}
		\item[$ii)$] The Morse index of $\Sigma$ in $(S^3,g)$ is at most one. 
		\item[$iii)$] If $\Sigma_k$ is stable in $(S^3,g_k)$ for every large enough $k$, then $m=1$ and $\Sigma$ is stable as well.
		\item[$iv)$] If $m=1$, then the sequence $\{\Sigma_k\}$ converges to $\Sigma$ as a one-sheeted graph over $\Sigma$ in the $C^{q-1}$ topology.
		\item[$v)$] If $m=2$, then, over the complement of a single point of $\Sigma$, the sequence $\{\Sigma_k\}$ converges to $\Sigma$ as a two-sheeted graph in the $C^{q-1}$ topology.  Moreover, $\Sigma$ admits a positive Jacobi function and is therefore stable.
	\end{itemize}
\end{thmA}

\begin{rmkA}
The case described in item $v)$ is ruled out if one knows, for instance, that the limit three-sphere $(S^3,g)$ does not contain an embedded minimal two-sphere admitting Jacobi functions, or does not contain an embedded stable minimal two-sphere of area less than $C$. In these situations, the convergence of $\Sigma_k$ to $\Sigma$ is graphical with multiplicity $m=1$.
\end{rmkA}

\subsection{The weak-$*$ closure of the convex hull of a positive cone of Radon measures} Let $X$ be a compact topological space that is Hausdorff and second-countable. Let $C^{0}(X)$ denote the set of continuous real functions on $X$, endowed with the sup-norm $||f||_0 = \sup \{f(x)\in \mathbb{R};\, x\in X\}$ for every $f$ in $C^{0}(X)$. Under these assumptions on the space $X$, $C^{0}(X)$ is a separable Banach space. \\
\indent Let $\mathcal{M}(X)$ denote the set of (positive) Radon measures on $X$. The Riesz Representation Theorem identifies $\mathcal{M}(X)$ with the convex positive cone in the dual of $C^{0}(X)$ (= the set of all bounded linear operators on $C^{0}(X)$) consisting of all those operators $L$ such that $L(f)\geq 0$ for every non-negative function $f$ in $C^{0}(X)$. In fact, for every such linear map, there exists a (positive) Radon measure such that 
\begin{equation*}
	L(f) = \int_{X} fd\mu \quad \text{for all} \quad f\in C^{0}(X).
\end{equation*}
Moreover, $\mu$ is unique (up to the natural equivalence relation on the set of measures) and satisfies $\mu(X)=||L||= \sup \{L(f)\in \mathbb{R};\,||f||_0\leq 1 \}$ (see \cite{Sim}). \\
\indent The weak-$*$ topology on $\mathcal{M}(X)$ is the coarsest topology such that, for every $f$ in $C^{0}(X)$, the evaluation map $$ \mu \in \mathcal{M}(X) \mapsto \int_{X}fd\mu \in \mathbb{R}$$ is continuous. Under the identification of $\mathcal{M}(X)$ with a subset of the dual of $C^{0}(X)$, the weak-$*$ topology on $\mathcal{M}(X)$ is nothing but the induced weak-$*$ topology from the dual of $C^{0}(X)$. Under our assumptions on $X$, the weak-$*$ topology on $\mathcal{M}(X)$ is metrisable. In particular, compactness can be expressed in terms of sequences, and a measure $\mu$ belongs to the weak-$*$ closure of a subset $\mathcal{Y}$ of $\mathcal{M}(X)$ if and only if it there exists a sequence of measures $\mu_k$ in $\mathcal{Y}$ such that
\begin{equation*}
	\lim_{k\rightarrow +\infty} \int_{X} f d\mu_k = \int_{X} fd\mu \quad \text{for all} \quad f\in C^{0}(X).
\end{equation*}
\indent In the above setting, the following general statement holds:
\begin{thmB} \label{thm-abstract-equidistribution}
	Let $\mathcal{Y}$ be a non-empty weak-$*$ compact subset of $\mathcal{M}(X)$. The following assertions about a measure $\mu_0$ in $\mathcal{M}(X)$ are equivalent to each other:
	\begin{itemize}
		\item[$i)$] For every function $f$ in $C^{0}(X)$ such that $\int_{X}fd\mu_0 < 0$, there exists $\mu\in \mathcal{Y}$ such that $\int_{X}f d\mu \leq 0$.
		\item[$ii)$] For every function $f$ in $C^{0}(X)$ such that $\int_{X}fd\mu_0 = 0$, there exists $\mu\in \mathcal{Y}$ such that $\int_{X}f d\mu \leq 0$.
		\item[$iii)$] $\mu_0$ belongs to the weak-$*$ closure of the convex hull of the positive cone over $\mathcal{Y}$.
		\item[$iv)$] There exists a sequence $\{\mu_k\}$ in $\mathcal{Y}$ such that
		\begin{equation} \label{eq-appendix-equidistribution}
			\lim_{k\rightarrow +\infty} \frac{1}{k} \sum_{i=1}^{k}\frac{1}{\mu_i(X)}\int_{X}f d\mu_{i} = \frac{1}{\mu_{0}(X)}\int_{X} f d\mu_{0} \quad \text{for all} \quad f\in C^{0}(X).
		\end{equation}
	\end{itemize}
\end{thmB}
\begin{proof}
\indent $i) \Rightarrow ii)$: if $f$ in $C^{0}(X)$ is such that $\int_{X}fd\mu_0 = 0$, then for every positive integer $k$, the function $f_k = f-1/k\in C^{0}(X)$ is such that $\int_{X} f_k d\mu_0 < 0$. By assumption, there exists $\mu_k$ in $\mathcal{Y}$ such that $\int_{X}f_k d\mu_k \leq 0$. Since $\mathcal{Y}$ is weak-$*$ compact, passing to a subsequence we may assume that $\mu_{k_i}$ converges weakly to some measure $\mu$ in $\mathcal{Y}$. Since $f_{k_i}$ converges uniformly to $f$, it follows that $\int_{X} fd\mu = \lim\int_{X} f_{k_i} d\mu_{k_i} \leq 0$.\\
\indent $ii) \Rightarrow iii)$: by contradiction, assume that $\mu_0$ does not belong to $\mathcal{K}=$ the weak-$*$ closure of the convex hull of the positive cone over $\mathcal{Y}$. Notice that $\mathcal{K}$ is a weak-$*$ closed convex set. By the Hahn-Banach geometric separation theorem applied to the dual space of $C^{0}(X)$ endowed with the weak-$*$ topology (see \cite{Bre}, Theorem I.7), there exists $f $ in $C^{0}(X)$ such that 
\begin{equation*}
	\sup_{\mu\in \mathcal{K}}\int_{X} fd\mu  <  \int_{X} fd\mu_0.
\end{equation*}
\noindent Since $\mathcal{K}$ is a positive cone, this inequality implies that $\sup_{\mathcal{K}} \int_{X} fd\mu = 0$. In particular,
\begin{equation*}
 \int_{X} fd\mu_0 > 0 \quad \text{and} \quad \int_{X} fd\mu \leq 0 \quad \text{for all} \quad \mu \in \mathcal{Y}.
\end{equation*} 
\noindent Let $f_0 = \frac{1}{\mu_0(X)}\int_{X} fd\mu_0 - f$. Then $\int_{X} f_{0}d\mu_{0} = 0$ and, for every $\mu$ in $\mathcal{Y}$,
\begin{equation*}
 \int_{X} f_0 d\mu = \frac{\mu(X)}{\mu_0(X)}\int_{X} fd\mu_0 - \int_{X} f d\mu > 0,
\end{equation*}
\noindent a contradiction with the assumption $ii)$. \\
\indent $iii) \Rightarrow iv)$ Given a positive measure $\mu$ in $\mathcal{M}(X)$, we denote by $\overline{\mu}$ its normalisation:
\begin{equation*}
	\int_{X} fd \overline{\mu} = \frac{1}{\mu(X)}\int_{X} fd\mu \quad \text{for all} \quad f\in C^{0}(X).
\end{equation*}
\noindent Clearly, a measure $\mu_{0}$ belongs to the weak-$*$ closure of the convex hull of the positive cone over $\mathcal{Y}$ if and only is its normalisation $\overline{\mu}_{0}$ belongs to the weak-$*$ closure of the convex hull of the positive cone over $\overline{\mathcal{Y}}=\{\overline{\mu} \in \mathcal{M}(X);\, \mu \in \mathcal{Y}\}$. Without loss of generality, we can therefore assume that $\mu_0$ and all elements of $\mathcal{Y}$ are normalised: their total mass is equal to one. \\
\indent The assumption $iii)$ means that, for every positive integer $k$ there exists a positive integer $N_k$, a collection of real numbers $\alpha_{k,1},\ldots,\alpha_{k,N_k} \in [0,1]$ that add up to one, a collection of positive real numbers $\lambda_{k,1},\ldots,\lambda_{k,N_k}$, and a collection of measures $\mu_{k,1},\ldots,\mu_{k,N_k}$ in $\mathcal{Y}$ that satisfy the following property: for every $f\in C^{0}(X)$, the sequence
\begin{equation*}
	\varepsilon_k = \left\vert \sum_{i=1}^{N_k} \alpha_{k,i} \int_{X} fd(\lambda_{k,i}\mu_{k,i}) - \int_{X} f d\mu_{0}\right\vert
\end{equation*}
converges to zero as $k$ goes to infinity. \\
\indent The argument now follows closely what is explained in \cite{MarNevSon} in the proof of their main theorem. The first step is to show that, for each $k$, the positive numbers $\alpha_{k,i}\lambda_{k,i}$ can be replaced by positive rational numbers with the same denominator. In order to see this, observe that, by substituting the constant function $f=1$, we obtain $$\lim_{k\rightarrow +\infty}  \sum_{i=1}^{N_k}\alpha_{k,i}\lambda_{k,i} = 1. $$
Next, for each $k$, choose positive integers  $d_{k}$ and $c_{k,1},\ldots,c_{k,N_k}$ such that 
\begin{equation*}
	\left\vert \alpha_{k,i}\lambda_{k,i}  - \frac{c_{k,i}}{d_{k}}  \right\vert < \frac{\varepsilon_{k}}{N_k} \quad \text{for each} \quad {i=1,\dots,N_k}.
\end{equation*}
\noindent Then,
\begin{align*}
	\left\vert 1 - \frac{1}{d_k}\sum_{i=1}^{N_{k}} c_{k,i} \right| & \leq  \left\vert 1 - \sum_{i=1}^{N_k} \alpha_{k,i}\lambda_{k,i}\right\vert + \sum_{k=1}^{N_k}\left\vert \alpha_{k,i}\lambda_{k,i}  - \frac{c_{k,i}}{d_{k}}  \right\vert \\ & < \left\vert 1 - \sum_{i=1}^{N_k} \alpha_{k,i}\lambda_{k,i}\right\vert + \varepsilon_k \rightarrow 0 \quad \text{as} \quad k\rightarrow +\infty.
\end{align*}
Moreover, as all measures $\mu_{k,i}$ are normalised,
\begin{multline*}
	\left\vert \frac{1}{d_k} \sum_{i=1}^{N_{k}}c_{k,i}\int_{X} f d\mu_{k,i} - \int_{X} f d\mu_{0} \right| \\ \leq \left\vert \sum_{i=1}^{N_{k}} \frac{c_{k,i}}{d_{k}}\int_{X} f d\mu_{k,i} - \sum_{i=1}^{N_k} \alpha_{k,i}\lambda_{k,i}\int_{X} f d\mu_{k,i} \right\vert \\ + \left\vert \sum_{i=1}^{N_k} \alpha_{k,i}\lambda_{k,i}\int_{X} f d\mu_{k,i} - \int_{X} f d\mu_{0}\right\vert \\
	\leq ||f||_{0}\sum_{i=1}^{N_k}\left\vert  \frac{c_{k,i}}{d_{k}}  - \alpha_{k,i}\lambda_{k,i} \right\vert + \varepsilon_{k} \\
	\leq (1+||f||_0)\varepsilon_k \rightarrow 0 \quad \text{as} \quad k\rightarrow +\infty.
\end{multline*}
Combining the two previous estimates, we conclude that the positive integer numbers $c_{k,i}$ are such that
\begin{equation*}
	\lim_{k\rightarrow +\infty} \frac{1}{\sum_{i=1}^{N_{k}} c_{k,i}} \sum_{i=1}^{N_{k}} c_{k,i}\int_{X}f d\mu_{k,i}
	 = \int_{X} fd\mu_0.
\end{equation*}
\indent The second step consists to show that it is possible to select some of the normalised measures $\mu_{k,l}$ (possibly repeating them), independently of the function $f$, in such way that formula \eqref{eq-appendix-equidistribution} holds for every $f$ in $C^{0}(X)$. This is a combinatorial argument; the details are precisely those explained in \cite{MarNevSon}, pages 14-16, and we refer the reader to their paper. \\
\indent $iv) \Rightarrow i)$: since \eqref{eq-appendix-equidistribution} holds for a sequence $\{\mu_{k}\}$ in $\mathcal{Y}$, if $f$ in $C^{0}(X)$ is such that $\int_{X} fd\mu_{0} < 0$, then $\int_{X} fd\mu_k < 0$ for at least one element $\mu_k \in \mathcal{Y}$ of this sequence. 
\end{proof}
\begin{rmkA}
	Let $c$, $C$ and $D$ be positive constants. Assume that the compact set $\mathcal{Y}$ in Theorem \ref{thm-abstract-equidistribution} has a further structure and consists of Radon measures $\mu$ of the form
	\begin{equation*}
		\mu = \sum_{i=1}^{N}n_i\cdot \mu_i,
	\end{equation*}
	where the positive integers $N, n_1,\ldots, n_k$ are smaller than $D$ and the Radon measures $\mu_i$ belong to another subset $\mathcal{W}\subset \mathcal{Y}$ such that 
	\begin{equation*}
		0 < c \leq \eta(X) \leq C \quad \text{for all} \quad \eta\in \mathcal{W}.
	\end{equation*}
	\indent If a given Radon measure $\mu_0$ belongs to the weak-$*$ closure of the convex hull of the positive cone over $\mathcal{Y}$, then the first part of the proof of the implication $iii) \Rightarrow iv)$ in Theorem \ref{thm-abstract-equidistribution} shows that
\begin{equation*}
	\lim_{k\rightarrow +\infty} \frac{1}{\sum_{j=1}^{N_k} c_{k,j}\mu_{k,j}(X)} \sum_{j=1}^{N_k}c_{k,j}\int_{X}f d\mu_{k,j} = \frac{1}{\mu_{0}(X)} \int_{X} fd\mu_0
\end{equation*}
for all continuous functions $f$ on $X$, where the constants $c_{k,j}$ are positive integers and the measures $\mu_{k,j}$ belong to $\mathcal{Y}$. Given the structure of the Radon measures in $\mathcal{Y}$, this implies that there are positive integers $c_{k,j,m}$ and measures $\mu_{k,j,m}$ in $\mathcal{W}$ such that
\begin{multline*}
	\lim_{k\rightarrow +\infty} \frac{1}{\sum_{j=1}^{N_k} \sum_{m=1}^{M_{k,j}}c_{k,j,m}\mu_{k,j,m}(X)} \sum_{j=1}^{N_k}\sum_{m=1}^{M_{k,j}}c_{k,j,m}\int_{X}f d\mu_{k,j,m}\\ = \frac{1}{\mu_{0}(X)} \int_{X} fd\mu_0 \quad \text{for all} \quad f\in C^{0}(X).
\end{multline*}
This equation can be compared with the equation before (13) in \cite{MarNevSon}. From this point, the same abstract combinatorial argument of \cite{MarNevSon} applies as well and shows that there exists a sequence $\{\mu_i\}$ of Radon measures in $\mathcal{W}$ such that
\begin{equation*}
	\lim_{i\rightarrow \infty}\frac{1}{\sum_{i=1}^{k}\mu_i(X)}\sum_{i=1}^{k}\int_{X} fd\mu_{i} = \frac{1}{\mu_{0}(X)}\int_{X} fd\mu_0 \quad \text{for all} \quad f\in C^{0}(X).
\end{equation*}
(\textit{Cf}. \cite{MarNevSon}, pages 14-16. We remark that the uniform lower bound for the total measure of elements in $\mathcal{W}$ is used in this step).
\end{rmkA}

\end{document}